\documentclass[12pt,reqno]{amsart}
\usepackage{amsmath}
\usepackage{float}
\usepackage{geometry}
\usepackage{etoolbox}
\usepackage{amscd, amsfonts, amssymb, graphicx, color}
\usepackage{url}
\usepackage{tikz-cd}
\usepackage{hyphenat}                                 \usepackage{mathtools}
\usepackage[colorlinks=true]{hyperref}
\usepackage{titlesec}
\usepackage{cite}
\titleformat{\section}
{\normalfont\bfseries\large}{\thesection}{1em}{}

\makeatletter
\patchcmd\maketitle
{\uppercasenonmath\shorttitle}
{}
{}{}
\patchcmd\maketitle
{\@nx\MakeUppercase{\the\toks@}}
{\the\toks@}
{}
{}{}
\patchcmd\@settitle{\uppercasenonmath\@title}{\Large}{}{}
\patchcmd\@setauthors
{\MakeUppercase{\authors}}
{\authors}
{}{}
\makeatother
\newtheorem{theorem}{Theorem}[section]

\newtheorem{lemma}{Lemma}[section]
\newtheorem{definition}{Definition}[section]
\newtheorem{remark}{Remark}[section]

\newtheorem{corollary}{Corollary}[section]
\newtheorem{Proof of Theorem}{Proof}

\hypersetup{urlcolor=blue, citecolor=red, linkcolor= blue}

\makeatletter
\renewcommand\subsection{\@startsection{subsection}{2}%
	\z@{.7\linespacing\@plus\linespacing}{.5\linespacing}%
	{\normalfont\bfseries}}
\makeatother

\begin{document}
	\title[$q$-Numerical Radius Estimates in Semi- Hilbertian Spaces]{$q$-Numerical Radius Estimates in Semi-Hilbertian Spaces and Their Relations with Matrix Means for Sectorial Matrices}
 \author{Jyoti Rani}
	\address{ndian Institute of Science Education and Research (IISER) Mohali, Knowledge City, S.A.S Nagar,
Punjab 140306}
	\email{jyotirani@iisermohali.ac.in}   
	
	
	
	\begin{abstract} 
In this paper, the $q$-numerical radius of operators in semi-Hilbertian spaces is studied. New characterizations are established, and sharp upper and lower bounds for the $q$-numerical radius are derived. Moreover, several inequalities involving operator monotone functions and matrix means for the $q$-numerical radius of sectorial matrices are obtained. 
	\end{abstract}
	
	\maketitle

\section{Introduction and Preliminaries}
Let $\mathcal{H}$ be a complex, separable, infinite-dimensional Hilbert space with inner product $\langle \cdot, \cdot \rangle$, and let $\|\cdot\|$ denote the associated norm. Let $\mathcal{B}(\mathcal{H})$ denote the algebra of bounded linear operators on $\mathcal{H}$. Suppose $A \in \mathcal{B}(\mathcal{H})$ is a positive operator, that is, $\langle Ax, x \rangle \ge 0$ for every $x \in \mathcal{H}$, and assume that $\dim(R(A)) \ge 2$. For any operator $T \in \mathcal{B}(\mathcal{H})$, we denote its range and null space by $R(T)$ and $N(T)$, respectively. Let $P_{\overline{R(A)}}$ be the orthogonal projection onto the closure of $R(A)$.
The operator $A$ induces a semi-inner product on $\mathcal{H}$ defined by
\begin{equation*}
\langle x, y \rangle_A = \langle Ax, y \rangle, \quad x,y \in \mathcal{H},
\end{equation*}
and the corresponding seminorm is given by
\begin{equation*}
\|x\|_A = \sqrt{\langle x, x \rangle_A}.
\end{equation*}
The vector space $\mathcal{H}$, equipped with this semi-inner product, is called a semi-Hilbertian space.
This semi-inner product induces an inner product on the quotient space $\mathcal{H}/N(A)$ given by
\begin{equation*}
[\bar{x}, \bar{y}] = \langle x, y \rangle_A, \quad \bar{x}, \bar{y} \in \mathcal{H}/N(A).
\end{equation*}
The completion of $(\mathcal{H}/N(A), [\cdot, \cdot])$ is isometrically isomorphic to the Hilbert space $\mathcal{R}(A^{1/2}) = (R(A^{1/2}), (\cdot, \cdot))$, where the inner product is defined by
\begin{equation*}
(A^{1/2}x, A^{1/2}y) = \langle P_{\overline{R(A)}}x, P_{\overline{R(A)}}y \rangle, \quad x,y \in \mathcal{H}.
\end{equation*}
Furthermore, for $T \in \mathcal{B}(\mathcal{H})$, if there exists a constant $c > 0$ such that
\begin{equation*}
\|Tx\|_A \le c \|x\|_A \quad \text{for all } x \in \overline{R(A)},
\end{equation*}
then the $A$-operator seminorm of $T$, denoted by $\|T\|_A$, is defined as
\begin{equation*}
\|T\|_A = \sup_{\substack{x \in \overline{R(A)} \ x \ne 0}} \frac{\|Tx\|_A}{\|x\|_A}.
\end{equation*}

Let $\mathcal{B}_{A^{1/2}}(\mathcal{H})$ be the set of all operators admitting $A^{\frac{1}{2}}$-adjoint. By Douglas Theorem \cite{douglus}, we have
\begin{equation*}
	\mathcal{B}_{A^{1/2}}(\mathcal{H}) = \{ T \in \mathcal{B}(\mathcal{H}) : \exists c > 0 \text{ such that } \|Tx\|_A \leq c\|x\|_A \ \forall x \in \mathcal{H} \}.
\end{equation*}
An operator $T \in {B}_{A^{1/2}}(\mathcal{H})$ is known as an $A$-bounded operator, and if $T \in \mathcal{B}_{A^{1/2}}(\mathcal{H})$ then $T(N(A)) \subseteq N(A)$ \cite{sen2024note}. Moreover, $\mathcal{B}_A(\mathcal{H})$ and $\mathcal{B}_{A^{1/2}}(\mathcal{H})$ are sub-algebras of $\mathcal{B(H)}$ and $\mathcal{B}_A(\mathcal{H})\subseteq\mathcal{B}_{A^{1/2}}(\mathcal{H}) \subseteq \mathcal{B(H)}$. Also, $\mathcal{B}_A(\mathcal{H})=\mathcal{B}_{A^{1/2}}(\mathcal{H})= \mathcal{B(H)}$, if $A$ is one-one and $R(A)$ is closed in $\mathcal{H}$.
The $A$-spectral radius of $T$ is defined as
$
r_A(T)=\lim _{n \rightarrow \infty}\left\|T^n\right\|_A^{1 / n}
$\cite[Theorem 1]{feki2020spectral}.

Let $T \in \mathcal{B}(\mathcal{H})$. An element $x \in \mathcal{H}$ is said to be $A$-orthogonal to $y \in \mathcal{H}$ if $\langle x, y \rangle_A = 0$ and this is denoted by $x \perp_A y$. For a subset $\mathcal{W} \subseteq \mathcal{H}$, define

\begin{equation*}
\mathcal{W}^{\perp_A} = \{ x \in \mathcal{H} : \langle x, w \rangle_A = 0 \ \text{for all } w \in \mathcal{W} \}.
\end{equation*}

An operator $W \in \mathcal{B}(\mathcal{H})$ is called an $A$-adjoint of $T$ if
\begin{equation*}
\langle Tx, y \rangle_A = \langle x, Wy \rangle_A \quad \text{for all } x, y \in \mathcal{H}.
\end{equation*}
The class of all operators that admit an $A$-adjoint is denoted by $\mathcal{B}_A(\mathcal{H})$. 
For $T \in \mathcal{B}_A(\mathcal{H})$, the operator equation $AX = T^*A$ admits a unique reduced solution, denoted by $T^{\sharp_A}$, satisfying $R(T^{\sharp_A}) \subseteq R(A)$.

An operator $T \in \mathcal{B}_A(\mathcal{H})$ is called $A$-self-adjoint if $AT$ is self-adjoint.
Note that $A$-self-adjointness of $T$ does not necessarily imply $T = T^{\sharp_A}$. However, one has
$
T = T^{\sharp_A}$ if and only if $T$ is $A$-self-adjoint and $R(T) \subseteq \overline{R(A)}$.
The operator $T$ is said to be $A$-positive, written $T \geq_A 0$, if $AT$ is a positive operator. In particular, every $A$-self-adjoint operator belongs to $\mathcal{B}_A(\mathcal{H})$. Moreover, if $T \in \mathcal{B}_A(\mathcal{H})$ is $A$-self-adjoint, then so is $T^{\sharp_A}$, and
\begin{equation*}
\bigl(T^{\sharp_A}\bigr)^{\sharp_A} = T^{\sharp_A}.
\end{equation*}
An operator $U \in \mathcal{B}_A(\mathcal{H})$ is called $A$-unitary if
\begin{equation*}
\|Ux\|_A = \|U^{\sharp_A}x\|_A = \|x\|A \quad \text{for all } x \in \mathcal{H}.
\end{equation*}
In this case, one has
\begin{equation*}
U^{\sharp_A} U = (U^{\sharp_A}) U^{\sharp_A} = P_{\overline{R(A)}}.
\end{equation*}
Furthermore, if $U$ is $A$-unitary, then so is $U^{\sharp_A}$, and
$
\|U\|_A = \|U^{\sharp_A}\|_A = 1.
$

The concept of the $A$-$q$-numerical range and the $A$-$q$-numerical radius was introduced in \cite{feki2026joint}. In the sequel, we assume that $q\in \mathcal{D}$, where  $\mathcal{D}=\{z \in \mathbb{C}: |z|\le 1\}$ denotes the closed unit disc in the complex plane. For a subspace $\mathcal{V}$ of $\mathcal{H}$, we denote by $\mathbb{S}_{q, A}(\mathcal{V})$ the subset of $\mathcal{V} \times \mathcal{V}$ given by
\begin{equation*}
	\mathbb{S}_{q, A}(\mathcal{V})=\left\{ (x,y)\in \mathcal{V}\times \mathcal{V} : \Vert x \Vert_{A}=\Vert y \Vert_{A}=1 \textup{ and } \langle x, y \rangle_{A}=q \right\}.
\end{equation*}
\begin{definition}
 For $T \in \mathcal{B}_{A^{1 / 2}}(\mathcal{H})$, the $A$-$q$-numerical range $W_{q,A}(T)$ and the $A$-$q$-numerical radius $w_{q,A}(T)$ of $T$, are respectively defined as follows:
\begin{align*}
 W_{q,A}(T)&=\left\{\langle T x, y\rangle_A: x \in \mathcal{H},(x,y) \in \mathbb{S}_{q, A}(\mathcal{H}) \right\},~\text{and}\\ 
 w_{q,A}(T)&=\sup \left\{|\lambda|: \lambda \in W_{q,A}(T)\right\} .
\end{align*} 
\end{definition}	

The above definition provides a unified framework that encompasses several important concepts in the literature. In particular:
\begin{itemize}
    \item[(i)] When $q = 1$, it reduces to the $A$-numerical range.
    \item[(ii)] When $A = I$, it reduces to the $q$-numerical range.
    \item[(iii)] When $A = I$ and $q = 1$, it coincides with the classical numerical range.
\end{itemize}
These concepts have been widely studied in the literature; see, for instance, \cite{rani2025q1,rani2025q,feki2026joint,zamani2019numerical, arias2008partial, stankovic2024some, gau2021numerical, kittaneh2026estimation}.

Let ${M}_n^{+}$ denote the set of all positive matrices.
The matrix mean $A \sigma B$ is a matrix mean on ${M}_n^{+}$ satisfying the following properties \cite{bedrani2021positive}. 
\begin{itemize}
    \item[(i)] If $A \leq C$ and $B \leq D$, then 
    \begin{equation}\label{mean}
        A \sigma B \leq C \sigma D,
    \end{equation}
    for any $A, B, C, D \in {M}_n^{+}$.
    \item[(ii)] For any $A, B \in {M}_n^{+}$ and any invertible $C \in {M}_n$, $C^*(A \sigma B) C=\left(C^* A C\right) \sigma\left(C^* B C\right)$.

    \item[(iii)] If $A_k \downarrow_k A$ and $B_k \downarrow_k B$, then $\left(A_k \sigma B_k\right) \downarrow_k (A \sigma B)$ for any $A_k, B_k, A, B \in {M}_n^{+}$.
    \item[(iv)] ${I}\sigma {I}={I}$.
\end{itemize}

In the work by Dury \cite{drury2015principal}, the matrix geometric mean of two accretive matrices $A$ and $B$ is defined by the following expression,
\begin{equation*}
   A \# B = \left( \frac{2}{\pi}\int_0^\infty (tA+t^{-1}B)^{-1}\frac{dt}{t}\right)^{-1}. 
\end{equation*}
Subsequently, in \cite{raissouli2017relative}, an extension of the aforementioned definition to the weighted geometric mean $A \#_t B$ for $0 < t < 1$ was introduced.
In \cite{bedrani2021positive}, the concept of matrix mean for two accretive matrices $A$ and $B \in M_n$ has been established as follows,
\begin{equation*}
    A \sigma_fB= \int_0^1(A!_sB)dv_f(s),
\end{equation*}
where $A!_sB=((1-s)A^{-1} +sB^{-1})^{-1}, s\in(0,1)$ is the weighted harmonic mean of $A,B$, and $f:(0,\infty) \to (0,\infty)$ is an operator monotone function satisfying $f(1)=1$ and $v_f$ is a probability measure characterizing $\sigma_f$. Furthermore, in \cite{bedrani2021positive}, characterization of the operator monotone function for an accretive matrix was provided. Let $A \in M_n$ be accerative and $f:(0,\infty) \to (0,\infty)$ be an operator monotone function with $f(1)=1$, then
\begin{equation*}
    f(A)=\int_0^1 ((1-s)I+sA^{-1})^{-1}dv_f(s),
\end{equation*}
where $v_f$ is the probability measure satisfying $f(x)=\int_0^1((1-s)+sx^{-1})^{-1}dv_f(s)$.
The logarithmic mean of accretive matrices $A$, $B$ was defined as follows, 
\begin{equation*}
    \mathcal{L}(A,B) =\int_0^1 A\#_tB dt
\end{equation*}
\cite{tan2019logarithmic}, where $A\#_tB$ is the weighted geometric mean of the accritive matrices $A,B$ defined in \cite{raissouli2017relative}. In \cite{mao2020inequalities}, the Heinz mean is defined as
\begin{equation*}
    \mathcal{H}_t(A,B)=\frac{A \#_tB+A \#_{1-t}B}{2}, ~~0 < t < 1.
\end{equation*}
For more information, one can refer to \cite{raissouli2017relative,bedrani2021positive,lin2014singular,lin2016some}.
Let $\mathcal{F}$ denotes the collection of all functions $f$, $f:(0, \infty) \to (0, \infty)$ is an operator monotone function and $f(1)=1$. 

Let $A \in M_n$, then $A$ is said to be sectorial if $W(A)$ is a subset of a sector $S_\alpha$ for some $\alpha \in \mathclose{[}0,\frac{\pi}{2}\mathopen{)}$  where
\begin{equation*}
	S_{\alpha}= \{ z \in \mathbb{C} : \mathcal{R} z >0 , |\mathcal{I}z| \le \tan(\alpha)(\mathcal{R} z)\}.
\end{equation*}
The class of all $n \times n$ sectorial matrices where $W(T) \subseteq S_{\alpha}$ is denoted by $\prod_{s,\alpha}^n$. The significance of sectorial matrices arises from their numerical ranges falling within a specific sector in the right half of the complex plane. This class of matrices, particularly useful in stability analysis, becomes an essential tool for comprehending and studying the behavior of dynamical systems. Researchers have extensively explored the characteristics of the numerical radius and numerical range of sectorial matrices due to their wide-ranging applications.
For a detailed review, one may refer the articles \cite{alakhrass2021sectorial,bedrani2021numerical,sammour2022geometric,alakhrass2020note}. The following results are important. 
\begin{lemma}\cite{bedrani2021positive}\label{normr}
    If $ A \in \prod_{s,\alpha}^n$, $f \in \mathcal{F}$, we have
    \begin{equation*}
   f(\|\mathcal{R}(A)\|) \le \|\mathcal{R}f(A)\|\le \sec^2(\alpha) f(\|\mathcal{R}(A)\|).    
    \end{equation*}
\end{lemma}
\begin{lemma}\cite{bedrani2021positive}\label{normr2}
    If $ A \in \prod_{s,\alpha}^n$, $f \in \mathcal{F}$, we have
    \begin{equation*}
   f(\mathcal{R}(A)) \le \mathcal{R}(f(A))\le \sec^2(\alpha) f(\mathcal{R}(A)).    
    \end{equation*}
\end{lemma}

 For unitarily invariant norm, we have 
\begin{equation}\label{function}
    |||f(A+B)||| \le |||f(A)+f(B)|||,~\text{where}~ A,B >0, f \in \mathcal{F}
\end{equation} \cite{ando1999norm}.
Moreover, one can note that $q$-numerical radius is not a unitarily invariant norm.

\section{$A$-$q$-Numerical Radius Estimations}

Consider $q \in \mathbb{C}$ with $|q|\le 1$ and $x \in\mathcal{H}$ with $\|x\|_A=1$. For any $z \in \mathcal{H}$ satisfying $\langle x,z \rangle_A=0$ and $\|z\|_A=1$, let $y=\overline{q}x+\sqrt{1-|q|^2}z$. This construction ensures $\|y\|_A=1$ and $\langle x,y \rangle_A=q$. Conversely, for any $y \in \mathcal{H}$ with $\|y\|_A=1$ and $\langle x,y \rangle_A=q$, set $z=\frac{1}{\sqrt{1-|q|^2}}(y-\overline{q}x)$, resulting in $\langle x,z \rangle_A=0$ and $\|z\|_A=1$. Thus, there exists a one-to-one correspondence between such a $z$ and $y$. Before proceeding to the next result, one can note that $\Re_A(T)$ and $\Im_A(T)$ are $ A$-self-adjoint operators.
	\begin{lemma}\label{la-normal}\cite{rani2025q1}
	If $T\in \mathcal{B}_A({\mathcal{H}})$, $T$ is $A$-self-adjoint operator and $q \in \mathcal{D}$, then 
	\begin{equation}\label{ea-normal}
		|q| w_A(T) \le w_{q,A}(T) \le w_A(T).
	\end{equation}
\end{lemma}

The following result was recently obtained in \cite{rani2025q1}. Here, an alternative proof is provided using the $A$-Cartesian decomposition.
\begin{theorem}\label{s1t4.1}
	Let $T\in \mathcal{B}_A({\mathcal{H}})$ and $q \in \mathcal{D}\setminus \{0\}$. Then
	\begin{equation}\label{e1}
		\frac{|q|}{2}\|T\|_A\le w_{q,A}(T)\le \|T\|_A.
	\end{equation}
	
\end{theorem}
\begin{proof}
	By Cartesian decomposition of $T$, we have
	\begin{equation*}
		\|T\|_A\le \|\Re_A(T)\|_A+\|\Im_A(T)\|_A.
	\end{equation*}
	Since $\Re_A(T)$ and $\Im_A(T)$ both are $A$-self-adjoint operators, Lemma \ref{la-normal} gives us
	\begin{equation*}
		\|T\|_A	\le \frac{1}{|q|}\left( w_{q,A}(\Re_A(T))+w_{q,A}(\Im_A(T))\right).
	\end{equation*} 	
	It is easy to check that $w_{q,A}(\Re_A(T))\le w_{q,A}(T)$ and $w_{q,A}(\Im_A(T))\le w_{q,A}(T)$. Therefore,
	\begin{equation*}
		\|T\|_A\le \frac{2}{|q|}w_{q,A}(T).
	\end{equation*}
\end{proof}
\begin{remark}
 If we take $q=1$ in inequality \eqref{e1}, we have
$
        \frac{1}{2}\|T\|_A \le w_{A}(T) \le \|T\|_A,
$
 which is same as corollary 2.8 \cite{zamani2019numerical}.   
 If we take $A=I$ and $q \in (0,1)$ in Theorem \ref{s1t4.1}, we have
			\begin{equation}\label{qin}
			\frac{q}{2}\|T\|\le w_{q}(T)\le \|T\|,
		\end{equation}
        which is obtained in \cite{stankovic2025some}.

\end{remark}
Since, 
\begin{equation}\label{eq2.1}
    \frac{\|T\|}{2}\le w(T) \le \|T\|.
\end{equation}
An improvement of the inequalities in \eqref{eq2.1} has been given in \cite{kittaneh2005numerical}. It says that for $T \in B(\mathcal{H})$, we have
\begin{equation}\label{k5}
	\frac{1}{4}\|T^*T+TT^*\| \le w^2(T) \le \frac{1}{2} \|T^*T+TT^*\|.
\end{equation}	
In our forthcoming result, we will generalize \eqref{k5} to $A$-$q$-numerical radius of $T$, where $q \in \mathcal{D}\setminus \{0\}$.
\begin{theorem}\label{nt4.25}
	Let $T \in \mathcal{B}_A(\mathcal{H})$ and $q \in \mathcal{D}\setminus \{0\}$. Then
	\begin{equation}\label{eq1.16}
		\frac{|q|^2}{4}	\|T^\#T+TT^\#\|_A  \le w_{q,A}^2(T) \le \frac{(2-|q|^2+4|q|\sqrt{1-|q|^2})}{2}\|TT^\#+T^\#T\|_A.
	\end{equation}
\end{theorem}
\begin{proof}
	Since $\Re_A(T)$ and $\Im_A(T)$ are $A$-self adjoint operators, by Lemma \ref{la-normal}, we have $|q|\|\Re_A(T)\|_A \le w_{q,A}(\Re_A(T))$ and $|q|\|\Im_A(T)\|_A \le w_{q,A}(\Im_A(T))$. Therefore,
	\begin{eqnarray*}
		\|\Re_A(T)\|_A \le \frac{1}{|q|}w_{q,A}(T),~\text{and}~ \|\Im_A(T)\|_A \le \frac{1}{|q|}w_{q,A}(T).
	\end{eqnarray*}
	Now,
	\begin{align*}
		\|T^\#T+TT^\#\|_A=&2\|\Re_A(T)^2+\Im_A(T)^2\|_A\\
		\le & 2\left(\|\Re_A(T)\|^2+\|\Im_A(T)\|^2\right)\\
        \le & 2 \left(\frac{1}{|q|^2}w_{q,A}^2(T)+ \frac{1}{|q|^2}w_{q,A}^2(T)\right).
	\end{align*}
	Thus we have,
	\begin{equation*}
		\frac{|q|^2}{4}\|TT^\#+T^\# T\|_A
		 \le w_{q,A}^2(T).
	\end{equation*}
	Now, 
	\begin{align*}
		|\langle T^\#x,y \rangle_A|^2=&| \langle (\Re_A(T^\#)+i \Im_A(T^\#))x,y \rangle_A |^2 \\
		\le & |\langle \Re_A(T^\#)x,y \rangle_A+i \langle \Im_A(T^\#)x,y \rangle_A|^2.
	\end{align*}
 
Taking $y=\bar{q}x+\sqrt{1-|q|^2}z$, we have
\begin{align*}
		&|\langle T^\#x,y \rangle_A|^2 \le  |\langle \Re_A(T^\#)x,\bar{q}x+\sqrt{1-|q|^2}z \rangle_A+i \langle \Im_A(T^\#)x,\bar{q}x+\sqrt{1-|q|^2}z \rangle_A|^2\\
		\le & \left( |q| \left| \langle (\Re_A(T^\#)+i\Im_A(T^\#))x,x \rangle_A \rangle_A\right|+\sqrt{1-|q|^2}\left| \langle \Re_A(T^\#)x,z \rangle_A+i \langle \Im_A(T^\#)x,z \rangle_A\right| \right)^2\\
		=&|q|^2\left| \langle (\Re_A(T^\#)+i\Im_A(T^\#))x,x \rangle_A\right|^2+(1-|q|^2)\left| \langle \Re_A(T^\#)x,z \rangle_A+i \langle \Im_A(T^\#)x,z \rangle_A\right|^2\\
		+ &2|q|\sqrt{1-|q|^2}\left| \langle \Re_A(T^\#)x,x \rangle_A+i \langle \Im_A(T^\#)x,x \rangle_A\right|\left| \langle \Re_A(T^\#)x,z \rangle_A+i \langle \Im_A(T^\#)x,z \rangle_A\right|\\
 \le & |q|^2	\left( \langle (\Re_A(T^\#)+i\Im_A(T^\#))x,x \rangle_A \right)\overline{\left( (\Re_A(T^\#)+i\Im_A(T^\#))x,x \rangle_A \right)}	\\
 +&(1-|q|^2)\left| \langle \Re_A(T^\#)x,z \rangle_A+i \langle \Im_A(T^\#)x,z \rangle_A\right|^2\\
 + &2|q|\sqrt{1-|q|^2}\left| \langle \Re_A(T^\#)x,x \rangle_A+i \langle \Im_A(T^\#)x,x \rangle_A\right|\left| \langle \Re_A(T^\#)x,z \rangle_A+i \langle \Im_A(T^\#)x,z \rangle_A\right|\\
 \le&|q|^2 \left( \left| \langle \Re(T^\# x,x \rangle \right|^2+\left| \langle \Im(T^\# x,x \rangle \right|^2 \right)\\
 +&(1-|q|^2)\left(\left| \langle \Re_A(T^\#)x,z \rangle_A\right|+\left| \langle \Im_A(T^\#)x,z \rangle_A\right|\right)^2\\
 + &2|q|\sqrt{1-|q|^2}\left(\left| \langle \Re_A(T^\#)x,x \rangle_A \right|+\left| \langle \Im_A(T^\#)x,x \rangle_A\right|\right) \left(\left| \langle \Re_A(T^\#)x,z \rangle_A \right|+\left| \langle \Im_A(T^\#)x,z \rangle_A\right|\right)\\
  \end{align*}
  Using Cauchy Schwarz's inequality, we have
  \begin{align*}
|\langle T^\#x,y \rangle_A|^2 \le& |q|^2\left( \| \Re_A(T^\#)x\|^2+\| \Im_A(T^\#)x\|^2\right)\\
+&(1-|q|^2) \left(\| \Re_A(T^\#)x\|+ \| \Im_A(T^\#)x\| \right)^2\\
+&2|q|\sqrt{1-|q|^2}\left(\| \Re_A(T^\#)x\|+ \| \Im_A(T^\#)x\| \right)^2\\
\le & \left(|q|^2+2-2|q|^2+2|q|\sqrt{1-|q|^2}\right)\left( \| \Re_A(T^\#)x\|^2+\| \Im_A(T^\#)x\|^2\right).
	\end{align*}		
Thus, we have
\begin{align*}
	|\langle T^\#x,y \rangle_A|^2		
		=&(2-|q|^2+4|q|\sqrt{1-|q|^2})\frac{\|(T^\#)^\# T^\#+T^\#(T^\#)^\#\|}{2}\\
		=&(2-|q|^2+4|q|\sqrt{1-|q|^2})\frac{ \|\left(T T^\#+T^\#T\right)^\#\|}{2}\\
			=&(2-|q|^2+4|q|\sqrt{1-|q|^2})\frac{ \|T T^\#+T^\#T\|}{2}.
	\end{align*}
Last inequality follows from $\|T\|=\|T^\#\|$.	Hence the result.
\end{proof}

\begin{remark}
If $A=I$ and $q=1$, then we get the well known inequality \eqref{k5}.
Take $q=1$ in \eqref{eq1.16}, we have
 \begin{equation}
		\frac{1}{4}	\|T^\#T+TT^\#\|_A  \le w_{A}^2(T) \le \frac{1}{2}\|TT^\#+T^\#T\|_A.
	\end{equation}  
\end{remark}

Based on the theorem mentioned above, we can derive the subsequent significant result.
\begin{corollary}
	Let $T \in \mathcal{B}_A(\mathcal{H})$, $q \in \mathcal{D}$ and $w_{q,A}^2(T)=\frac{|q|^2}{4}\|T^\#T+TT^\#\|_A$. Then 
	$$|q|^2\|\Re_A(e^{i\theta}T)\|_A^2=|q|^2\|\Im_A(e^{i\theta}T)\|_A^2=\frac{|q|^2}{4}\|T^\#T+TT^\#\|_A$$ for all $\theta \in \mathbb{R}$.
\end{corollary}	
\begin{proof}
	By simple calculations, we obtain that
	$$(\Re_A(e^{i\theta}T))^2+(\Im_A(e^{i\theta}T))^2=\frac{TT^\#+T^\#T}{2},$$ where $\theta \in \mathbb{R}$. 
	Now,
	\begin{align*}
		\frac{|q|^2}{4}\|T^\#T+TT^\#\|_A
		=&\frac{|q|^2}{2}\|	(\Re_A(e^{i\theta}T))^2+(\Im_A(e^{i\theta}T))^2\|_A\\
		\le & \frac{|q|^2}{2}(\|	(\Re_A(e^{i\theta}T))^2\|_A+\|(\Im_A(e^{i\theta}T))^2\|_A)\\
  \le & \frac{|q|^2}{2}(\|	(\Re_A(e^{i\theta}T))\|_A^2+\|(\Im_A(e^{i\theta}T))\|_A^2)\\
		\le & \frac{1}{2}(w_{q,A}^2(\Re_A(e^{i\theta}T))+w_{q,A}^2(\Im_A(e^{i\theta}T)))\\
		\le & \frac{1}{2}(w_{q,A}^2(e^{i\theta}T)+w_{q,A}^2(e^{i\theta}T))\\
		=&\frac{1}{2}(w_{q,A}^2(T)+w_{q,A}^2(T))\\
  =&w_{q,A}(T)\\
		=& \frac{|q|^2}{4}\|T^\#T+TT^\#\|_A.
	\end{align*}
	Thus, $$|q|^2\|\Re_A(e^{i\theta}T)\|_A^2=|q|^2\|\Im_A(e^{i\theta}T)\|_A^2=w_{q,A}^2(T)=\frac{|q|^2}{4}\|T^\#T+TT^\#\|_A$$ for all $\theta \in \mathbb{R}$.
\end{proof}

\begin{theorem}\label{th1.15}
	Let $T \in \mathcal{B}_A(\mathcal{H})$ and $q \in \mathcal{D}$. Then
	\begin{equation*}
		w_{q,A}(T)\leq(|q|^2w_A^2(T)+ 
		\left((1-|q|^2)\|T\|_A^2+2|q| \sqrt{1-|q|^2}w_A(T)\|T\|_A\right)^{\frac{1}{2}}.
	\end{equation*}
\end{theorem} 
\begin{proof}
	Consider,
	\begin{align*}
		|\langle Tx,y \rangle_A|^2=&|\langle Tx,\overline{q}x+ \sqrt{1-|q|^2}z \rangle_A|^2\\
		=&|\overline{q} \langle Tx,x \rangle_A + \sqrt{1-|q|^2} \langle Tx,z \rangle_A|^2\\
		\le & (|\overline{q} \langle Tx,x \rangle_A| +| \sqrt{1-|q|^2} \langle Tx,z \rangle_A|)^2\\
		=& |q|^2 | \langle Tx,x \rangle_A|^2 +(1-|q|^2)|\langle Tx,z \rangle_A|^2+2|q|\sqrt{1-|q|^2}| \langle Tx,x \rangle_A|| \langle Tx,z \rangle_A|\\
		\le & |q|^2 |\langle Tx,x \rangle_A|^2 +(1-|q|^2)|\langle Tx,z \rangle_A|^2+2|q| \sqrt{1-|q|^2}  |\langle Tx,x \rangle_A|\|T\|_A\|x\|_A\|z\|_A.
	\end{align*}
	This implies, 
	\begin{align*}
		|\langle Tx,y \rangle_A|^2 
		\le& |q|^2w_A^2(T)+ 
		(1-|q|^2)\|T\|_A^2+2|q| \sqrt{1-|q|^2}w_A(T)\|T\|_A.
	\end{align*}
	Taking supremum over $x$ and $y$ with $\|x\|=\|y\|=1$ and $\langle x,y \rangle =q$, we have 	
	\begin{equation*}
		w_{q,A}(T)\leq(|q|^2w_A^2(T)+ 
		\left((1-|q|^2)\|T\|_A^2+2|q| \sqrt{1-|q|^2}w_A(T)\|T\|_A\right)^{\frac{1}{2}}.
	\end{equation*}
	
	Hence the result.
\end{proof}
\begin{remark}
Take $A=I$ in Theorem \ref{th1.15}, we have
\begin{equation}\label{qe1}
		w_{q}(T)\leq(|q|^2w^2(T)+ 
		\left((1-|q|^2)\|T\|^2+2|q| \sqrt{1-|q|^2}w(T)\|T\|\right)^{\frac{1}{2}}.
	\end{equation}
\end{remark}

We can note that $w_{q,A}(.)$ is not multiplicative, namely the equality 
\begin{equation}\label{in4.7}
	w_{q,A}(TS) = w_{q,A}(T)w_{q,A}(S), ~ q \in \mathcal{D}\setminus \{0\}
\end{equation}
is not valid in general. To see this, consider
$T=\begin{bmatrix}
	0 & 1 \\
	1 & 0 
\end{bmatrix}$ and $S=\begin{bmatrix}
	2 & 0 \\
	0 & 2 
\end{bmatrix}$. Then $w_{q,A}(T)=|q|,w_{q,A}(S)=2|q|$ and $w_{q,A}(TS)=2|q|$, where $q \in \mathcal{D}$. Clearly, $w_q(TS) \ne w_q(T)w_q(S)$ if $q \in \mathcal{D}\setminus \{0,1\}$. Here, $T$ and $S$ commutes. It is worthy to note that even if $T$ and $S$ commutes, equality (\ref{in4.7}) does not hold true. While it holds true that $w(TS) \le 4w(T)w(S)$, it is important to note that $w_{q,A}(TS) \not \le 4w_{q,A}(T)w_{q,A}(S)$ for all $q \in \mathcal{D}\setminus \{0,1\}$.  	
\begin{theorem}
	Let $T \in \mathcal{B}_A(\mathcal{H})$ and $S \in \mathcal{B}_A(\mathcal{H})$ and $q \in \mathcal{D}\setminus \{0\}$. Then
	\begin{equation*}
		|q|^2w_{q,A}(TS)\le 4w_{q,A}(T)w_{q,A}(S). 
	\end{equation*}
\end{theorem}
\begin{proof}
	We have,
	\begin{align*}
		w_{q,A}(TS) \le& \|TS\|_A \\
		\le & \|T\|_A\|S\|_A.
	\end{align*}
	By using Theorem \ref{s1t4.1}, we obtain that
	\begin{equation*}
		w_{q,A}(TS)\le \frac{4}{|q|^2}w_{q,A}(T)w_{q,A}(S). 
	\end{equation*}
	Hence, the result.
\end{proof}
If we take $A=I$ and $q=1$ in above theorem then we obtain a well known inequality $w(TS)\le 4w(T)w(S)$.
\section{Realtion between $q$-Numerical Range and Matrix Means for Sectorial Matrices}
Next, we present $q$-numerical radius analog of some known results on operator monotone functions and matrix means.
Norm inequalities and classical numerical radius inequalities for matrix means are studied in \cite{bedrani2021positive} and \cite{bedrani2021numerical} respectively.
In this sub-section, we extend some existing results related to matrix means for $q$-numerical range using the results obtained in the previous sub-section.  In the next theorem, we give some relations between $f(A)$ and $w_q(A)$, where $f \in \mathcal{F}$. In particular, we present $q$-numerical radius version of inequality (\ref{function}) for sectorial matrices. Also, if $A$ is sectorial then $f(A)$ is also sectorial \cite[Corollary 3.4]{bedrani2021numerical}. The following results is useful in this direction.
\begin{lemma}\label{co1.6}\cite{rani2025q}
    If $T \in \mathcal{B(H)}$ and $q \in \mathcal{D}$, then we have 
		\begin{itemize}
			\item [(a)] $|q|\|\mathcal{R}(T)\| \le w_q(\mathcal{R}(T))\le \|\mathcal{R}(T)\|$,
			\item [(b)]  $|q|\|\mathcal{I}(T)\| \le w_q(\mathcal{I}(T))\le \|\mathcal{I}(T)\|.$
		\end{itemize}
\end{lemma}
\begin{lemma}\label{normr1}\cite{zhang2015matrix}
 Let  $ A \in \prod_{s,\alpha}^n$ and $|||.|||$ be any unitarily invariant norm on $M_n$. Then
 \begin{equation*}
    \cos({\alpha})|||A|||\le ||| \mathcal{R}(A)||| \le |||A|||. 
 \end{equation*}
\end{lemma}
\begin{lemma}\label{realcor}
If $A \in \prod_{s,\alpha}^n$ and $q \in \mathcal{D'}$. Then
\begin{itemize}
    \item [(a)]
    $\cos(\alpha)w_q(A) \le \|\mathcal{R}(A)\|,$

\item [(b)]
$|q| \cos(\alpha)w_q(A) \le w_q(\mathcal{R}(A)).$
\end{itemize}
\end{lemma}
\begin{proof}
\begin{itemize}
\item [(a)]Using Lemma \ref{normr1}, we have
\begin{equation*}
   w_q(A) \le \|A\| \le \sec(\alpha) \| \mathcal{R}(A)\|. 
\end{equation*}
Thus,
\begin{equation}\label{3.4}
    \cos(\alpha)w_q(A) \le \| \mathcal{R}(A)\|.
\end{equation}
\item [(b)] Part (b) follows from Corollary \ref{co1.6}(a) and relation (\ref{3.4}).
\end{itemize}

\end{proof}
\begin{theorem}\label{opmo}
If $ A \in \prod_{s,\alpha}^n$, $q \in \mathcal{D'}$ and $f \in \mathcal{F}$ then
\begin{itemize}
    \item [(a)] $|q|^2\cos{\alpha}f(w_q(A))\le |q|w_q(f(A)) \le \sec^3(\alpha)f(w_q(A)),$
\item [(b)]$|q|w_q((1-\gamma)f(A) +\gamma f(B)) \le \sec^3(\alpha)f((1-\gamma) w_q(A)+\gamma w_q(B))$, where $\gamma \in (0,1)$.
\item [(c)]$|q|w_q(f(A+B)) \le \sec^3(\alpha) w_q(f(A)+f(B)).$
\end{itemize}
\end{theorem}
\begin{proof}
\begin{itemize}
    \item [(a)]
The relation 
\begin{equation*}\label{alpha}
    |q|\cos(\alpha)f(w_q(A)) \le |q|f ( \cos(\alpha)w_q(A))
\end{equation*}
holds true as $f(rx) \ge r f(x)$ for all $0 \le r \le 1$. Also,
 \begin{align*}
   |q|f (  \cos(\alpha)w_q(A)) \le & |q|f(\|\mathcal{R}(A)\|)\\
   \le & |q|\| \mathcal{R}f(A) \|\\
   = & \| |q|\mathcal{R}f(A) \|\\
   \le & w_q( \mathcal{R}f(A))\\
    \le & w_q(f(A)),
 \end{align*}
 where the first, second, and fourth inequalities follow from Lemma \ref{realcor}, Lemma \ref{normr}, and Lemma \ref{co1.6}
 respectively. Hence
 \begin{equation*}
     |q|\cos(\alpha)f(w_q(A))\le w_q(f(A)).
 \end{equation*}
 Again
 \begin{align*}
|q|w_q(f(A)) \le & 
|q| \|f(A)\|\\
\le & |q| \sec(\alpha)\|\mathcal{R}(f(A))\|\\
\le& |q| \sec^3(\alpha)f(\| \mathcal{R}(A) \|)\\
\le & \sec^3(\alpha)f(|q|\| \mathcal{R}(A) \|)\\
\le  &\sec^3(\alpha)f(w_q( \mathcal{R}(A)))\\
\le & \sec^3(\alpha)f(w_q(A)),
\end{align*}
where the first and second inequalities follow from Lemma \ref{normr1} and Lemma \ref{normr} respectively and the third inequality follow from the fact that $f(rx) \ge r f(x)$ for all $0 \le r \le 1$.
This proves the result.
 \item [(b)] For $\gamma \in (0,1)$, we have 
 $$w_q((1-\gamma)f(A)+\gamma f(B)) \le (1-\gamma)w_q(f(A))+\gamma w_q(f(B)). $$
 By using part (a), we have 
 $$w_q((1-\gamma)f(A)+\gamma f(B)) \le \frac{1}{|q|}\sec^3(\alpha)((1-\gamma)f(w_q(A))+\gamma f(w_q(B))).$$
 Concavity of $f$ implies that 
 $$|q|w_q((1-\gamma)f(A) +\gamma f(B)) \le \sec^3(\alpha)f((1-\gamma) w_q(A)+\gamma w_q(B)).$$
 \item [(c)] To prove the inequality, we have
 
 By using Lemma \ref{normr1}, Lemma \ref{normr2} and relation (\ref{function}), respectively, we have
 \begin{align*}
   w_q(f(A+B)) \le & \|f(A+B)\|\\
   \le & \sec(\alpha)\|\mathcal{R}f(A+B)\|\\
   \le & \sec^3(\alpha)\|f(\mathcal{R}(A+B))\|\\
   = & \sec^3(\alpha)\|f(\mathcal{R}(A)+\mathcal{R}(B))\|\\
   \le & \sec^3(\alpha)\|f(\mathcal{R}(A))+f(\mathcal{R}(B))\|\\
   \le & \sec^3(\alpha)\|\mathcal{R}(f(A)+f(B))\|\\
    \le & \frac{1}{|q|}\sec^3(\alpha)w_q(\mathcal{R}(f(A)+f(B))\\
    \le & \frac{1}{|q|}\sec^3(\alpha)w_q(f(A)+f(B)),
  \end{align*}
  where the second and third inequalities follow from Lemma \ref{normr1} and Lemma \ref{normr2} respectively, the fifth inequality follows from the relation (\ref{function}), and the sixth and seventh inequalities follow from Lemma \ref{normr2} and Corollary \ref{co1.6}.
 \end{itemize}
\end{proof}

Based on the above result, we obtain the following important Corollaries.
In particular, if $f(x)=x^t$; $t \in (0,1)$, we obtain the following result.
\begin{corollary}\label{cort}
If $ A \in \prod_{s,\alpha}^n$, $q \in \mathcal{D'}$, and $t \in(0,1)$ then
\begin{itemize}
    \item [(a)] $|q|^2 \cos(\alpha)w^t_q(A)\le |q|w_q(A^t) \le \sec^3(\alpha)w^t_q(A)$,
    \item [(b)] $|q|w_q( (1-\gamma) A^t+ \gamma B^t ) \le \sec^3(\alpha)((1-\gamma)w_q(A)+\gamma w_q(B))^t$, $\gamma \in (0,1)$,
    \item [(c)] $|q|w_q((A+B)^t) \le \sec^3(\alpha) w_q(A^t+B^t)$.
\end{itemize}
\end{corollary}
\begin{remark}
     If we take $\gamma= \frac{1}{2}$ in Corollary \ref{cort}(b), we have
  \begin{equation}\label{ine}
     |q|w_q(A^t+B^t) \le 2^{1-t}\sec^3(\alpha)(w_q(A)+w_q(B))^t, 
  \end{equation}
  which gives us subadditive behaviors of $q$-numerical radius for non integral powers $t \in(0,1)$.
 By Corollary \ref{cort}(c)  and inequality \eqref{ine}, we have
 \begin{equation*}
   |q| \cos^3(\alpha)w_q((A+B)^t) \le w_q(A^t+B^t) \le \frac{1}{|q|} 2^{1-t}\sec^3(\alpha)(w_q(A)+w_q(B))^t. 
 \end{equation*}
\end{remark}
If $A$ is a positive matrix then by taking $\alpha=0$ in Theorem \ref{opmo}, we have the following corollary.
\begin{corollary}\label{corpositive}
 If A is a positive matrix, $q \in \mathcal{D'}$ and $f \in \mathcal{F}$  then
 \begin{itemize}
     \item [(a)]$ |q|^2f(w_q(A))\le |q|w_q(f(A)) \le f(w_q(A))$,
     \item [(b)] $|q|w_q((1-\gamma)f(A) +\gamma f(B)) \le f((1-\gamma) w_q(A)+\gamma w_q(B))$,
     \item [(c)]$|q|w_q(A+B)^t\le w_q(A^t+B^t)$.
 \end{itemize}
\end{corollary}
\begin{remark}
 Taking $q=1$ in Corollary \ref{corpositive}(c), we get well known inequality $\|(A+B)^t\|\le \|A^t+B^t\|$ for positive matrices $A$ and $B$.
\end{remark}

In the next theorem, we will discuss the $q$-numerical radius inequalities of matrix means for sectorial matrices. 
\begin{theorem}
  Let $A,B,C,D \in \prod_{s,\alpha}^n$, $q \in \mathcal{D'}$ such that $\mathcal{R}(A) \le \mathcal{R}(C)$ and $\mathcal{R}(B) \le \mathcal{R}(D)$, then
  \begin{equation*}
|q|w_q(A \sigma_f B)\le \sec^3(\alpha)w_q(C \sigma_f D).      
  \end{equation*}
\end{theorem}
\begin{proof}
 Using Lemma \ref{realcor}(a), we have
 $$|q| \cos(\alpha)w_q(A \sigma_f B) \le |q|\| \mathcal{R}(A \sigma_f B)\|.$$
 Also, \cite[Theorem 5.3]{bedrani2021positive} implies that
 \begin{align*}
     |q| \cos(\alpha)w_q(A \sigma_f B) \le & |q|\sec^2(\alpha)\| \mathcal{R}(C \sigma_f D)\|\\
     \le & \sec^2(\alpha) w_q(C \sigma_f D).
 \end{align*}
 Hence, the required result.
\end{proof}
\begin{theorem}\label{tmean}
 Let $ A \in \prod_{s,\alpha}^n$, $q \in \mathcal{D'}$. Then
 \begin{equation*}
    |q|^2 w_q(A \sigma_f B) \le \sec^3(\alpha) (w_q(A) \sigma_f w_q(B)). 
 \end{equation*}
\end{theorem}
\begin{proof}
 By applying Lemma \ref{realcor}(a), \cite[Proposition 5.2]{bedrani2021positive}, relation (3.13) in \cite{ando1994majorizations} and relation (\ref{mean}), respectively the following inequalities are obtained,
 \begin{align*}
  w_q(A \sigma_f B) \le & \sec(\alpha) \|\mathcal{R}(A \sigma_f B)\|\\
  \le &\sec^3(\alpha)\| \mathcal{R}(A) \sigma_f  \mathcal{R}(B)\|\\
  \le &\sec^3(\alpha)\| \mathcal{R}(A)\| \sigma_f  \|\mathcal{R}(B)\|\\
  \le & \sec^3(\alpha) \left(\frac{1}{|q|}w_q(\mathcal{R}(A)) \right)\sigma_f \left(\frac{1}{|q|}w_q(\mathcal{R}(B)) \right)\\
   \le & \sec^3(\alpha) \left(\frac{1}{|q|}w_q(A)\right) \sigma_f \left(\frac{1}{|q|}w_q(B)\right).
 \end{align*}
 This completes the proof.
\end{proof}
\begin{remark}
 If $A$ and $B$ are positive matrices, we have
 \begin{eqnarray*}
  |q|^2w_q( A \sigma_f B) \le w_q(A) \sigma_f w_q(B),
 \end{eqnarray*}
 where $q \in \mathcal{D'}$.
 If $q=1$, then we get well known inequality $\| A \sigma_f B\| \le \|A\| \sigma_f \|B\|$ for positive matrices $A$ and $B$.
\end{remark}
\begin{corollary}\label{cor3.47}
  Let $ A \in \prod_{s,\alpha}^n$, $q \in \mathcal{D'}$ and $0<t<1$, we have
  \begin{itemize}
      \item [(a)]
      \begin{equation*}
  |q|^2w_q(A \#_t B ) \le \sec^3(\alpha) w_q^{1-t}(A) w_q^t(B).  
  \end{equation*}
  \item [(b)]
   \begin{equation*}
  |q|^2w_q((1-t)A^{-1} +t B^{-1} ) \le \sec^3(\alpha) ((1-t)w_q^{-1}(A)+ tw_q^{-1}(B)).  
  \end{equation*}
  \end{itemize}
  
\end{corollary}
\begin{proof}
\begin{itemize}
    \item [(a)]
      Taking $\sigma_f= \#_t$, Theorem \ref{tmean} follows
\begin{align*}
|q|^2w_q(A \#_t B ) \le & \sec^3(\alpha) (w_q(A) \#_t w_q(B))\\
\le & \sec^3(\alpha) w_q ^{1-t}(A) w_q^t(B).
\end{align*}
\item [(b)]
Taking $\sigma_f= !_t$ in Theorem \ref{tmean}, we have
\begin{align*}
|q|^2w_q(A !_t B ) \le & \sec^3(\alpha) (w_q(A) !_t w_q(B)).
\end{align*}
Hence, by definition of weighted harmonic mean, we have
\begin{equation*}
     |q|^2w_q((1-t)A^{-1} +t B^{-1} ) \le \sec^3(\alpha) ((1-t)w_q^{-1}(A)+ tw_q^{-1}(B)).
\end{equation*}
\end{itemize}
   
\end{proof}
If we take $t=\frac{1}{2}$, we have
\[ |q|^2w_q(A^{-1}+B^{-1}) \le \sec^2(\alpha)(w_q^{-1}(A)+w_q^{-1}(B).\]
Also, the logarithmic mean and Heinz mean satisfy the similar type of relations, as follows
\begin{theorem}
    Let $ A \in \prod_{s,\alpha}^n$, $q \in \mathcal{D'}$, $0<t<1$. Then
 \begin{itemize}
 \item [(a)] $|q|^2w_q(\mathcal{L}(A,B)) \le \sec^3(\alpha) \mathcal{L}(w_q(A), w_q(B))$.
 \item [(b)] $|q|^2w_q(\mathcal{H}_t(A,B))\le \sec^3(\alpha) \mathcal{H}(w_q(A), w_q(B)) $.
 \end{itemize}   
\end{theorem}
\begin{proof}
 \begin{itemize}
\item [(a)] $|q|^2w_q(\mathcal{L}(A,B)) \le w_q (\int_0^1 A \#_t B dt )\le  \int_0^1 w_q(A \#_t B )dt  $.
Using Corollary \ref{cor3.47}, we have
\begin{equation*}
w_q(\mathcal{L}(A,B)) \le \frac{\sec^3(\alpha)}{|q|^2}\int_0^1 w_q^{1-t}(A) w_q^t(B)dt.   
\end{equation*}
\item [(b)]
Using definition of $H_t(A,B)$ and corollary \ref{cor3.47} respectively, we have
\begin{align*}
   w_q(H_t(A,B))=&w_q \left(\frac{A\#_t B+A\#_{1-t} B}{2}\right)\\
   \le & \frac{w_q(A\#_t B)}{2}+\frac{w_q(A\#_{1-t} B)}{2}\\
   \le & \frac{\sec^3(\alpha)}{2|q|^2}(w_q^{1-t}(A)w_q^t(B)+w_q^{1-t}(B)w_q^t(A)\\
   \le & \frac{\sec^3(\alpha)}{2|q|^2} H_t(w_q(A),w_q(B)).
\end{align*}
 \end{itemize}   
\end{proof}
The next theorem gives us a relation among $w_q (A \#_t B)$, $w_q(H_t(A,B))$ and $w_q(A\Delta B)$, where $A\Delta B$ denotes the arithmetic mean of $A$ and $B$.
\begin{theorem}
   Let $ A \in \prod_{s,\alpha}^n$, $q \in \mathcal{D'}$, and $0<t<1$. Then
   \begin{equation*}
      |q|^2\cos^4(\alpha) w_q (A \# B)\le |q|w_q(H_t(A,B)) \le{\sec^4(\alpha)}w_q(A\Delta B). 
   \end{equation*}
\end{theorem}
\begin{proof}
By using \cite[Theorem 2.12 ]{mao2020inequalities} and Theorem 2.10 \cite{rani2025q} respectively, we have
  \begin{align*}
    w_q (A \# B) \le& \|A \# B\|\\
    \le& \sec^3(\alpha) \|H_t(A,B)\| \\
    \le & \frac{\sec^4(\alpha)}{|q|}w_q(H_t(A,B)).
\end{align*}
Again, By using \cite[Theorem 2.12 ]{mao2020inequalities} and Theorem 2.10 \cite{rani2025q} respectively, we obtain that
\begin{align*}
w_q(H_t(A,B)) \le & \|H_t(A,B)\| \\
\le & \frac{\sec^3(\alpha)}{2}\|A+B\|\\
\le & \frac{\sec^4(\alpha)}{2|q|}w_q(A+B)\\
=&\frac{\sec^4(\alpha)}{|q|}w_q(A\Delta B).
\end{align*}
This completes the proof.
\end{proof}
	\bibliographystyle{abbrv}
\bibliography{bib_NR_SP}

\begin{thebibliography}{10}

\bibitem{alakhrass2020note}
M.~Alakhrass.
\newblock A note on sectorial matrices.
\newblock {\em Linear Multilinear Algebra}, 68(11):2228--2238, 2020.

\bibitem{alakhrass2021sectorial}
M.~Alakhrass.
\newblock On sectorial matrices and their inequalities.
\newblock {\em Linear Algebra Appl.}, 617:179--189, 2021.

\bibitem{ando1994majorizations}
T.~Ando.
\newblock Majorizations and inequalities in matrix theory.
\newblock {\em Linear Algebra Appl.}, 199:17--67, 1994.

\bibitem{ando1999norm}
T.~Ando and X.~Zhan.
\newblock Norm inequalities related to operator monotone functions.
\newblock {\em Math. Ann.}, 315(4):771--780, 1999.

\bibitem{arias2008partial}
M.~L. Arias, G.~Corach, and M.~C. Gonzalez.
\newblock Partial isometries in semi-hilbertian spaces.
\newblock {\em Linear Algebra Appl.}, 428(7):1460--1475, 2008.

\bibitem{bedrani2021positive}
Y.~Bedrani, F.~Kittaneh, and M.~Sababheh.
\newblock From positive to accretive matrices.
\newblock {\em Positivity}, 25(4):1601--1629, 2021.

\bibitem{bedrani2021numerical}
Y.~Bedrani, F.~Kittaneh, and M.~Sababheh.
\newblock Numerical radii of accretive matrices.
\newblock {\em Linear Multilinear Algebra}, 69(5):957--970, 2021.

\bibitem{douglus}
R.~G. Douglas.
\newblock On majorization, factorization, and range inclusion of operators on
  hilbert space.
\newblock {\em Proc. Amer. Math. Soc.}, 17(2):413--415, 1966.

\bibitem{drury2015principal}
S.~Drury.
\newblock Principal powers of matrices with positive definite real part.
\newblock {\em Linear Multilinear Algebra}, 63(2):296--301, 2015.

\bibitem{feki2020spectral}
K.~Feki.
\newblock Spectral radius of semi-hilbertian space operators and its
  applications.
\newblock {\em Ann. Funct. Anal.}, 11:929--946, 2020.

\bibitem{feki2026joint}
K.~Feki, A.~Patra, J.~Rani, and Z.~Taki.
\newblock Joint q-numerical ranges of operators in hilbert and semi-hilbert
  spaces.
\newblock {\em Linear and Multilinear Algebra}, pages 1--23, 2026.

\bibitem{gau2021numerical}
H.-L. Gau and P.~Y. Wu.
\newblock {\em Numerical ranges of Hilbert space operators}, volume 179.
\newblock Cambridge University Press, 2021.

\bibitem{kittaneh2005numerical}
F.~Kittaneh.
\newblock Numerical radius inequalities for hilbert space operators.
\newblock {\em Studia Mathematica}, 168(1):73--80, 2005.

\bibitem{kittaneh2026estimation}
F.~Kittaneh, A.~Patra, and J.~Rani.
\newblock On the estimation of the q-numerical radius via orlicz functions.
\newblock {\em Journal of Computational and Applied Mathematics}, page 117340,
  2026.

\bibitem{lin2014singular}
M.~Lin.
\newblock Singular value inequalities for matrices with numerical ranges in a
  sector.
\newblock {\em Oper. Matrices}, 8:1143--1148, 2014.

\bibitem{lin2016some}
M.~Lin.
\newblock Some inequalities for sector matrices.
\newblock {\em Oper. Matrices}, 10(4):915--921, 2016.

\bibitem{mao2020inequalities}
Y.~Mao and Y.~Mao.
\newblock Inequalities for the heinz mean of sector matrices.
\newblock {\em Bull. Iran. Math. Soc.}, 46:1767--1774, 2020.

\bibitem{raissouli2017relative}
M.~Ra{\"\i}ssouli, M.~S. Moslehian, and S.~Furuichi.
\newblock Relative entropy and tsallis entropy of two accretive operators.
\newblock {\em Comptes Rendus Math.}, 355(6):687--693, 2017.

\bibitem{rani2025q}
J.~Rani and A.~Patra.
\newblock $ q $-numerical radius of sectorial matrices and $2 \times 2$
  operator matrices.
\newblock {\em arXiv preprint arXiv:2501.14505}, 2025.

\bibitem{rani2025q1}
J.~Rani, A.~Patra, and R.~Birbonshi.
\newblock On the $ a $-$ q $-numerical range of operators in semi-hilbertian
  spaces.
\newblock {\em arXiv preprint arXiv:2502.19855}, 2025.

\bibitem{sammour2022geometric}
S.~A. Sammour, F.~Kittaneh, and M.~Sababheh.
\newblock A geometric approach to numerical radius inequalities.
\newblock {\em Linear Algebra Appl.}, 652:1--17, 2022.

\bibitem{sen2024note}
A.~Sen, R.~Birbonshi, and K.~Paul.
\newblock A note on the a-numerical range of semi-hilbertian operators.
\newblock {\em Linear Algebra Appl.}, 703:268--288, 2024.

\bibitem{stankovic2024some}
H.~Stankovi{\'c}, M.~Krsti{\'c}, and I.~Damnjanovi{\'c}.
\newblock Some properties of the q-numerical radius.
\newblock {\em Linear Multilinear Algebra}, pages 1--22, 2024.

\bibitem{stankovic2025some}
H.~Stankovi{\'c}, M.~Krsti{\'c}, and I.~Damnjanovi{\'c}.
\newblock Some properties of the q-numerical radius.
\newblock {\em Linear and Multilinear Algebra}, 73(8):1736--1757, 2025.

\bibitem{tan2019logarithmic}
F.~Tan and A.~Xie.
\newblock On the logarithmic mean of accretive matrices.
\newblock {\em Filomat}, 33(15):4747--4752, 2019.

\bibitem{zamani2019numerical}
A.~Zamani.
\newblock A-numerical radius inequalities for semi-hilbertian space operators.
\newblock {\em Linear Algebra Appl.}, 578:159--183, 2019.

\bibitem{zhang2015matrix}
F.~Zhang.
\newblock A matrix decomposition and its applications.
\newblock {\em Linear Multilinear Algebra}, 63(10):2033--2042, 2015.

\end{thebibliography}

\end{document}